\documentclass[11pt]{article}
\usepackage{amssymb,amsmath}
\usepackage{float}
\usepackage{color,fullpage}
\usepackage{cite}
\usepackage[title,titletoc]{appendix}
\usepackage{algorithm,algorithmic}
\UseRawInputEncoding
\usepackage{amsthm}
\newcommand{\RR}{\mathbb{R}}

\newcommand{\prox}{{\mathbf{prox}}}

\newcommand{\cY}{{\mathcal{Y}}}
\newcommand{\cO}{{\mathcal{O}}}
\newcommand{\cD}{{\mathcal{D}}}

\newcommand{\cS}{{\mathcal{S}}}

\newcommand{\cZ}{{\mathcal{Z}}}

\newtheorem{assumption}{Assumption}[section]

\newtheorem{lemma}{Lemma}[section]
\newtheorem{definition}{Definition}[section]

\newtheorem{proposition}{Proposition}[section]

\makeatletter

\@addtoreset{equation}{section} \makeatother

\begin{document}

\title{Extragradient and Extrapolation Methods with Generalized Bregman Distances for Saddle Point Problems}

\author{Hui Zhang\thanks{
Department of Mathematics, National University of Defense Technology,
Changsha, Hunan 410073, China.  Email: \texttt{h.zhang1984@163.com}
}
}

\date{\today}

\maketitle

\begin{abstract}
In this work, we introduce two algorithmic frameworks, named Bregman extragradient method and Bregman extrapolation method, for solving saddle point problems. The proposed frameworks not only include the well-known extragradient and optimistic gradient methods as special cases, but also generate new variants such as sparse extragradient and extrapolation methods. With the help of the recent concept of relative Lipschitzness and some Bregman distance related tools, we are able to show certain upper bounds in terms of Bregman distances for gap-type measures. Further, we use those bounds to deduce the convergence rate of $\cO(1/k)$ for the Bregman extragradient and Bregman extrapolation methods applied to solving smooth convex-concave saddle point problems. Our theory recovers the main discovery made in [Mokhtari et al. (2020), SIAM J. Optim., 20, pp. 3230-3251] for more general algorithmic frameworks with weaker assumptions via a conceptually different approach.
\end{abstract}




\section{Introduction}
The extragradient method is a powerful tool for solving smooth convex-concave saddle point problems. Its original scheme was introduced by Korpelevich as early as 1976 in \cite{1976An}. During the past four decades, this method has been extensively developed from several aspects such as extending its range \cite{2004Prox,2015On,2018A} and simplifying its iterate \cite{2011The}.  Recently, due to the fact that many game and learning problems are actually equivalent to finding saddle points of min-max optimization problems, it has attracted increasing attention in the machine learning, computer science, and optimization communities.

At the algorithmic level, many works aim to modify the original extragradient method. Remarkable examples include Popov's modification of the {A}rrow-{H}urwicz method \cite{1980A}, Tseng's modified forward-backward splitting method \cite{1998A}, Nemirovski's prox-method \cite{2004Prox}, Malitsky's reflected projected method \cite{2015Projected} and its generalization in \cite{2020A,2020simple}. It should be pointed out that Malitsky's reflected projected method covers the optimistic gradient descent-ascent (OGDA) method, which recently appeared in machine learning for training GANs (see \cite{2018training}).

To guarantee convergence, the standard assumptions are monotonicity and the Lipschitz continuity. However, they are sufficient but not necessary; see e.g. \cite{2015On,2020simple,2020efficient}. Very recently, Cohen et al. in \cite{2020Relative} proposed the relative Lipschitzness of operator as an alternative to the standard assumptions. The relative Lipschitzness of function in first-order methods was original introduced in \cite{Bauschke2016A,Lu2016relatively} to go beyond the gradient Lipschitz continuity.

In this paper, we first introduce two algorithmic frameworks, namely Bregman extragradient method and Bregman extrapolation method, to unify existing extragradient-type algorithms. Then, by employing the relative Lipschitzness of operator, we deduce certain upper bounds in terms of Bregman distances for gap-type measures of the proposed algorithmic frameworks.  Applying to smooth convex-concave saddle point problems, we show that these algorithms converge with a rate $\cO(1/k)$. This demonstrates that our theory extends the recent discovery in \cite{2020Convergence}, which was made by formulating the extragradient and OGDA methods as approximations of the proximal point algorithm \cite{1976mono}, to more general algorithmic frameworks with weaker assumptions via a conceptually different approach.

The rest of the paper is organized as follows.  In Section \ref{se2}, we introduce some basic convex analysis and  Bregman distance related tools. In Section \ref{se3}, we propose two algorithmic frameworks with generalized Bregman distances as well as some specialized variants. In Section \ref{se4}, we list a group of assumptions and establish two main iterate results for the proposed algorithmic frameworks.
In Section \ref{se5}, we derive the convergence rate $\cO(1/k)$ for the proposed algorithmic frameworks applied to solving smooth convex-concave saddle point problems.  Concluding remarks are given in Section
\ref{se6}. Missing proofs are postponed to Appendix.

\section{Preliminaries}\label{se2}

\subsection{Notation}
In this paper, we restrict our analysis into real finite dimensional spaces $\RR^d$. We use $\langle \cdot, \cdot\rangle$ to denote the inner product and  $\|\cdot\|$ to denote the Euclidean norm. For a multi-variables function $f(x,y)$, we use $\nabla_xf$ (respectively, $\nabla_yf$) to denote the gradient of $f$ with respective to $x$ (respectively, $y$).

\subsection{Convex analysis tools}
We present some basic notations and facts about convex analysis, which will be used in our results.
\begin{definition}\label{sc0}
A function $\phi:\RR^d\rightarrow \RR$ is convex if for any $\alpha\in [0,1]$ and $u, v\in \RR^d$, we have
\begin{equation*}
\phi(\alpha u+(1-\alpha)v)\leq \alpha \phi(u)+(1-\alpha)\phi(v);
\end{equation*}
and strongly convex with modulus $\mu> 0$ if for any $\alpha\in [0,1]$ and $u, v\in \RR^d$, we have
\begin{equation*}
\phi(\alpha u+(1-\alpha)v)\leq \alpha \phi(u)+(1-\alpha)\phi(v)-\frac{1}{2}\mu\alpha(1-\alpha)\|u-v\|^2. \label{SC1}
\end{equation*}
Further, $\phi$ is concave if $-\phi$ is convex.
\end{definition}

\begin{definition}
Let $\phi:\RR^d\rightarrow \RR$ be a convex function. The subdifferential of $\phi$  at $u\in \RR^d$ is defined as
$$\partial \phi (u):= \{ u^* \in \RR^d: \phi(v)\geq \phi(u)+ \langle u^*, v-u\rangle,\quad \forall v\in \RR^d \}.$$
The elements of $\partial \phi(u)$ are called the subgradients of $\phi$ at $u$.
\end{definition}

The subdifferential generalizes of the classical concept of differential because of the well-known fact that $\partial \phi(u)=\{\nabla \phi(u)\}$ when the function $\phi$ is differentiable.
In terms of the subdifferential, the strong convexity in Definition \ref{sc0} can be equivalently stated as \cite{Hiriart2004}: For any $u, v\in \RR^d$ and $v^*\in \partial\phi(v)$, we have
\begin{equation}\label{sc01}
 \phi(u)\geq \phi(v)+\langle v^*, u-v\rangle +\frac{\mu}{2}\|u-v\|^2.
\end{equation}

\begin{definition}
Let $\phi:\RR^d\rightarrow \RR$ be a convex function. The conjugate of $\phi$ is defined as
$$\phi^*(u^*)=\sup_{v\in \RR^d}\{\langle u^*,v \rangle -\phi(v)\}.$$
\end{definition}

\begin{definition}
A function $\phi:\RR^d\rightarrow \RR$ is gradient-Lipschitz-continuous with modulus $L>0$ if for any $u, v\in \RR^d$, we have
\begin{equation*}
\|\nabla \phi(u)-\nabla \phi(v)\|\leq L\|u-v\|. \label{Lip}
\end{equation*}
\end{definition}

The following facts are well-known, which could be found from the classic textbooks \cite{1970convex} and \cite{Hiriart2004}.

\begin{lemma}\label{scLip}
Let $\phi:\RR^d\rightarrow \RR$ be a strongly convex function with modulus $\mu> 0$. Then we have that
 \begin{itemize}
   \item its conjugate $\phi^*$ is gradient-Lipschitz-continuous with modulus $\frac{1}{\mu}$;
   \item the conditions $\phi(u)+\phi^*(u^*)=\langle u, u^*\rangle$, $u^*\in \partial \phi(u)$, and $u\in\partial \phi^*(u^*)$ are equivalent.
 \end{itemize}

\end{lemma}

\subsection{Bregman distance tools}
The Bregman distance, originally introduced in \cite{1967The}, is a very powerful concept in many fields where distances are involved. Recently, many variants of Bregman distances appears, see e.g. \cite{1997Legendre,1997Free,2018Re}. For simplicity as well as generality, we choose the Bregman distance defined by a strongly convex.
\begin{definition}
Let $\omega:\RR^d\rightarrow \RR$ be a strongly convex function with modulus $\mu> 0$. The Bregman distance $D_\omega^{v^*}(u,v)$ between $u, v\in \RR^d$ with respect to $\omega$ and a subgradient $v^*\in\partial \omega(v)$ is defined by
\begin{equation}\label{Breg}
D_\omega^{v^*}(u,v):=\omega(u)-\omega(v)-\langle v^*, u-v\rangle.
\end{equation}
\end{definition}
In the following, we state three basic facts about the Bregman distance, which will be used later in our analysis. It should be pointed out that the results are well-known--see e.g. \cite{1997Free,1997proximal}. We list them here, along with a brief proof, for completeness.
\begin{lemma}\label{lemBreg}
Let $\omega:\RR^d\rightarrow \RR$ be a strongly convex function with modulus $\mu> 0$. For any $u, p, q\in \RR^d$ and $p^*\in \partial \omega (p), q^*\in \partial \omega (q)$, we have that
\begin{equation}\label{Bregdis1}
D_\omega^{p^*}(u,p) - D_\omega^{q^*}(u,q) + D_\omega^{q^*}(p,q)=\langle q^*-p^*, u-p\rangle,
\end{equation}
\begin{equation}\label{Bregdis2}
D_\omega^{q^*}(p,q)=D_{{\omega}^*}^p(q^*,p^*),
\end{equation}
and
\begin{equation}\label{Bregdis3}
D_\omega^{q^*}(p,q)\geq \frac{\mu}{2}\|p-q\|^2.
\end{equation}
\end{lemma}

\begin{proof}
The equality \eqref{Bregdis1} follows from \eqref{Breg}, and the inequality \eqref{Bregdis3} from \eqref{sc01} and \eqref{Breg}. To obtain \eqref{Bregdis2}, we derive that
\begin{eqnarray}
\begin{array}{lll}
D_\omega^{q^*}(p,q) &= & \omega(p)-\omega(q)-\langle q^*, p-q\rangle    \\
  &= &\langle p, p^*\rangle -\omega^*(p^*) -\langle q, q^*\rangle +\omega^*(q^*) -\langle q^*, p-q\rangle  \\
&= &\omega^*(q^*)-\omega^*(p^*) -\langle p, q^*-p^*\rangle\\
&=& D_{{\omega}^*}^p(q^*,p^*),
\end{array}
\end{eqnarray}
where the second and fourth lines follow by using the second part of Lemma \ref{scLip} and the condition $p^*\in \partial \omega (p), q^*\in \partial \omega (q)$.
\end{proof}

\section{Algorithmic frameworks}\label{se3}
In this section, we introduce two algorithmic frameworks, both of which are constructed by an operator $F:\RR^d\rightarrow \RR^d$ and a strongly convex function  $\omega:\RR^d\rightarrow \RR$ via certain coupled styles.

\subsection{Bregman extragradient method}
Let $u_0, u_0^*\in\RR^{d}$ and positive parameters $\{\alpha_k\}$ be given. The Bregman extragradient method, abbreviated as BEG, updates the iterates $\{u_k\}$ for $k\geq0$ via the following scheme.
\begin{eqnarray}\label{alg1}
\left\{\begin{array}{lll}
\bar{u}_k &= & \arg\min_{u\in\RR^{d}}\{\alpha_k\langle F(u_k), u \rangle +D_\omega^{u_k^*}(u,u_k)\},   \\
u_{k+1} &= & \arg\min_{u\in\RR^{d}}\{\alpha_k\langle F(\bar{u}_k), u \rangle +D_\omega^{u_k^*}(u,u_k)\},   \\
u_{k+1}^*&= &u_k^*- \alpha_k F(\bar{u}_k).
\end{array} \right.
\end{eqnarray}
Equivalently, it can be rewritten as
\begin{eqnarray}\label{alg12}
\left\{\begin{array}{lll}
\bar{u}_k &= & \nabla \omega^*(u_k^*- \alpha_k F(u_k)),   \\
u_{k+1} &= & \nabla \omega^*(u_k^*- \alpha_k F(\bar{u}_k)),    \\
u_{k+1}^*&= &u_k^*- \alpha_k F(\bar{u}_k).
\end{array} \right.
\end{eqnarray}
Due to the newly introduced ``parameter" $\omega$, BEG not only includes the standard extragradient method as its special case by taking $\omega(u)=\frac{1}{2}\|u\|^2$, but also generates implicitly regularized variants. To illustrate the latter, we take $\omega$ as the augmented $\ell_1$-norm \cite{2013Aug}, that is $\omega(u)=\gamma\|u\|_1+\frac{1}{2}\|u\|^2$, where $\gamma$ is a positive constant and $\|\cdot\|_1$ is the $\ell_1$-norm defined as the sum of absolute values of the entries. It is easy to see that the augmented $\ell_1$-norm is a strongly convex function with modulus $\mu=1$. Further, we have
\begin{equation}\label{shrink1}
\nabla \omega^*(\cdot)=\cS_\gamma(\cdot),
\end{equation}
where $\cS_\gamma(\cdot)$ is the well-known shrinkage operator defined by
 $$\cS_\gamma(u):=\textrm{sign}(u)\max\{|u|-\mu,0\},$$
with  $\textrm{sign}(\cdot), |\cdot|$, and $\max\{\cdot, \cdot\}$ being component-wise operations for vectors. Now, the BEG method for this special case becomes
\begin{eqnarray}\label{alg13}
\left\{\begin{array}{lll}
\bar{u}_k &= & \cS_\gamma(u_k^*- \alpha_k F(u_k)),   \\
u_{k+1} &= & \cS_\gamma(u_k^*- \alpha_k F(\bar{u}_k)),    \\
u_{k+1}^*&= &u_k^*- \alpha_k F(\bar{u}_k).
\end{array} \right.
\end{eqnarray}
Because the shrinkage operator $\cS_\gamma(\cdot)$ generates sparse vectors, we call the newly appeared scheme \eqref{alg13} sparse extragradient method. Its remarkable advantage is that the iterates may be sparse, although the shrinkage operations also increase computational cost. Thereby, how to balance these two sides is worthy of future studying. In addition, it should be noted that although the iterates could be sparse, their averaged sequences, whose iteration complexities will be studied, may be dense.

In slightly more general, let us consider the case of $\omega(u)=\psi(u)+\frac{1}{2}\|u\|^2$, where $\psi$ is a convex regularized function with an ``easily" computational proximal operator given by
$$\prox_\psi(u):=\arg\min_v\{\psi(v)+\frac{1}{2}\|u-v\|^2\}.$$
Substituting such $\omega$ into \eqref{alg1}, we immediately obtain the following regularized extragradient method
\begin{eqnarray}\label{alg14}
\left\{\begin{array}{lll}
\bar{u}_k &= & \prox_\psi(u_k^*- \alpha_k F(u_k)),   \\
u_{k+1} &= & \prox_\psi(u_k^*- \alpha_k F(\bar{u}_k)),    \\
u_{k+1}^*&= &u_k^*- \alpha_k F(\bar{u}_k).
\end{array} \right.
\end{eqnarray}
It seems that, to the best of our knowledge, such variants have not appeared before us.

\subsection{Bregman extrapolation method}
Let $u_0, u_{-1}, u_0^*\in\RR^{d}$, positive parameters $\{\alpha_k\}$, and nonnegative parameters $\{\beta_k\}$ be given. The Bregman extrapolation method, abbreviated as BEP, updates the iterates $\{u_k\}$ for $k\geq0$ via the following scheme.
\begin{eqnarray}\label{alg2}
\left\{\begin{array}{lll}
u_{k+1}  &= & \arg\min_{u\in\RR^{d}}\{\alpha_k\langle F(u_k)+\beta_k(F(u_k)-F(u_{k-1})), u \rangle +D_\omega^{u_k^*}(u,u_k)\}   \\
u_{k+1}^*&= &u_k^*- \alpha_k F(u_k)-\alpha_k\beta_k(F(u_k)-F(u_{k-1})).
\end{array} \right.
\end{eqnarray}
Equivalently, it can be rewritten as
\begin{eqnarray}\label{alg22}
\left\{\begin{array}{lll}
u_{k+1}&= &  \nabla \omega^*(u_k^*- \alpha_k F(u_k)-\alpha_k\beta_k(F(u_k)-F(u_{k-1}))),   \\
u_{k+1}^*&= &u_k^*- \alpha_k F(u_k)-\alpha_k\beta_k(F(u_k)-F(u_{k-1})).
\end{array} \right.
\end{eqnarray}
BEP is general enough to include several existing algorithms as its special cases. For example, BEP with $\omega=\frac{1}{2}\|\cdot\|^2$ and $\alpha_k\equiv\eta, \beta_k\equiv 1$ recovers the optimistic gradient descent ascent (OGDA) method; it also recovers the operator extrapolation method \cite{2020simple} by taking $\omega$ to be differentiable and the modified forward-backward splitting \cite{2020A} specialized to our setting. When considering the case of $\omega(u)=\gamma\|u\|_1+\frac{1}{2}\|u\|^2$, we have the following scheme, called sparse extrapolation method.
\begin{eqnarray}\label{alg23}
\left\{\begin{array}{lll}
u_{k+1}&= &  \cS_\gamma(u_k^*- \alpha_k F(u_k)-\alpha_k\beta_k(F(u_k)-F(u_{k-1}))),   \\
u_{k+1}^*&= &u_k^*- \alpha_k F(u_k)-\alpha_k\beta_k(F(u_k)-F(u_{k-1})).
\end{array} \right.
\end{eqnarray}
Corresponding to the case of  $\omega(u)=\psi(u)+\frac{1}{2}\|u\|^2$, the regularized extrapolation method reads as
\begin{eqnarray}\label{alg24}
\left\{\begin{array}{lll}
u_{k+1}&= &  \prox_\psi(u_k^*- \alpha_k F(u_k)-\alpha_k\beta_k(F(u_k)-F(u_{k-1}))),   \\
u_{k+1}^*&= &u_k^*- \alpha_k F(u_k)-\alpha_k\beta_k(F(u_k)-F(u_{k-1})).
\end{array} \right.
\end{eqnarray}
At last, we would like to point out that the BEP method does not cover the
simultaneous centripetal acceleration and alternating centripetal acceleration methods, proposed in our recent work \cite{2020Training} for training GANs, one of which also includes the OGDA as a special case.

\section{Iteration properties of the algorithmic frameworks}\label{se4}
In order to provide a unified convergence analysis for the previously introduced algorithmic frameworks, we make the following assumptions about the operator $F$ and the function $\omega$.

\begin{assumption}\label{Lip}
Let $\omega:\RR^d\rightarrow \RR$ be a strongly convex function with modulus $\mu>0$ and let $\lambda$ be a positive parameter.  The operator $F$ is $\lambda$-relatively Lipschitz with respect to $\omega$, i.e., for any $u, v, z\in \RR^d$, we have
   $$\langle F(v)-F(u),v-z\rangle \leq \lambda(D_\omega^{u^*}(v,u)+D_\omega^{v^*}(z,v)).$$
\end{assumption}

This assumption is a slight modification of the relative Lipschitzness recently proposed in \cite{2020Relative}, and will be a key tool in the forthcoming convergence analysis. As observed in \cite{2020Relative}, the relative Lipshcitzness is a more general condition encapsulating the standard Lipschitz assumption as well as the more recent relative smoothness assumption \cite{Bauschke2016A,Lu2016relatively}. The following, shown in \cite{2020Relative}, will be required later.

\begin{lemma}\label{resclip}
If $\omega$ is strongly convex with modulus $\mu$ and $F$ is L-Lipschitz in the sense that for any $u, v\in \RR^d$, we have $\|F(u)-F(v)\|\leq L\|u-v\|$, then $F$ is $L/\mu$-relatively Lipschitz with respect to $\omega$.
\end{lemma}

\begin{assumption}\label{nonemp}
The solution set of $F$ defined as $\cY:=\{u: F(u)=0\}$ is nonempty.
\end{assumption}

\begin{assumption}\label{mono}
The operator $F$ is monotone, i.e., for any $u, v\in \RR^d$  we have
   $$\langle F(u)-F(v),u-v\rangle \geq 0.$$
\end{assumption}

\begin{assumption}\label{coer}
The conjugate function $\omega^*:\RR^d\rightarrow \RR$ satisfies coercivity, i.e, for any fixed $v\in \RR^d$, we have
$$\omega^*(u^*)-\langle v, u^*\rangle\rightarrow +\infty,\quad \|u^*\|\rightarrow +\infty.$$
\end{assumption}
It is easy to verify that Assumption \ref{coer} holds for $\omega^*(u^*)=\frac{1}{2}\|u^*\|^2$. Actually, it holds for the conjugates of strongly convex functions, as a direct result of Proposition 14.15 in \cite{Bauschke2017}.
\begin{lemma}\label{coer1}
Let $\omega:\RR^d\rightarrow \RR$ be a strongly convex function with modulus $\mu>0$. Then the conjugate function $\omega^*$ satisfies the coercivity.
\end{lemma}

Now, we are ready to present the main results of this study. The first result is a upper bound of the ``regret measure" $\langle \alpha_k F(\bar{u}_k), \bar{u}_k-u \rangle$ by the difference between the generalized Bregman distances $D_\omega^{u_k^*}(u,u_k)$ and $D_\omega^{u_{k+1}^*}(u,u_{k+1})$, for the Bregman extragradient method.
\begin{proposition}\label{conEG}
Let $\{\bar{u}_k, u_k\}$ be the iterates generated by the Bregman extragradient method introduced in \eqref{alg1}. Suppose that Assumption \ref{Lip} holds and the parameters $\alpha_k$ satisfy $0<\lambda\alpha_k\leq 1$.
Then, we have
\begin{equation}\label{recur1}
\langle \alpha_k F(\bar{u}_k), \bar{u}_k-u \rangle \leq D_\omega^{u_k^*}(u,u_k) - D_\omega^{u_{k+1}^*}(u,u_{k+1}).
\end{equation}
Moreover, if Assumptions \ref{nonemp}-\ref{mono} also hold, then the sequence $\{\bar{u}_k\}$ is bounded.
\end{proposition}
The first part of this result is a modification of Lemma 3.1 in \cite{2004Prox} (see also Proposition 1 in \cite{2020Relative}), and the second part and the followings are partially inspired by the work \cite{2020Convergence}.

\begin{proposition}\label{conEP}
Let $\{u_k\}$ be the iterates generated by the Bregman extrapolation method introduced in \eqref{alg2} with the initial conditions $u_0=u_{-1}$. Suppose that Assumption \ref{Lip} holds and the parameters $\alpha_k, \beta_k$ satisfy the condition
\begin{eqnarray}\label{eq0}
\left\{\begin{array}{lll}
\alpha_k\beta_k=\alpha_{k-1}, \\
\lambda(\alpha_k+\alpha_{k-1})\leq 1.
\end{array} \right.
\end{eqnarray}
Then, we have
\begin{eqnarray}\label{eq01}
\begin{array}{lll}
 \alpha_k\langle F(u_{k+1}), u_{k+1}-u\rangle&\leq & \alpha_k\langle F(u_{k+1})-F(u_k), u_{k+1}-u\rangle-\alpha_{k-1}\langle F(u_k)-F(u_{k-1}), u_k-u\rangle\\
 &&+D_\omega^{u_k^*}(u,u_k) - D_\omega^{u_{k+1}^*}(u,u_{k+1})\\
 &&+\lambda\alpha_{k-1}D_\omega^{u_{k-1}^*}(u_k,u_{k-1}) - \lambda\alpha_kD_\omega^{u_k^*}(u_{k+1},u_k).
\end{array}
\end{eqnarray}
Moreover, if Assumptions \ref{nonemp}-\ref{mono} also hold and there is a positive constant $\rho$ such that $\lambda\alpha_k\leq 1-\rho$, then the sequence $\{u_k\}$ is bounded.
\end{proposition}

To illustrate the generality of the result above, we consider the OGDA method, i.e., the special case of BEP with $\alpha_k\equiv\eta, \beta_k\equiv1, \omega =\frac{1}{2}\|\cdot\|^2$. The condition on the step size $\eta$ becomes $0<\eta\leq \frac{1}{2L}$ due to \eqref{eq0} and $\mu=1, \lambda=\frac{L}{\mu}$ because of Lemma \ref{resclip}. Note that $D_\omega^{v^*}(u,v)=\frac{1}{2}\|u-v\|^2$. Under this setting, the existence of $\rho$ such that $\lambda\alpha_k\leq 1-\rho$ also holds. From \eqref{eq01}, we obtain
\begin{eqnarray}\label{eq001}
\begin{array}{lll}
\langle F(u_{k+1}), u_{k+1}-u\rangle&\leq & \langle F(u_{k+1})-F(u_k), u_{k+1}-u\rangle-\langle F(u_k)-F(u_{k-1}), u_k-u\rangle\\
 &&+\frac{1}{2\eta}\|u-u_k\|^2  -\frac{1}{2\eta}\|u-u_{k+1}\|^2  \\
 &&+\frac{L}{2} \|u_k-u_{k-1}\|^2 - \frac{L}{2}  \|u_{k+1}-u_k\|^2,
\end{array}
\end{eqnarray}
which is exactly Lemma 8(a) in \cite{2020Convergence}. The second part of Proposition \eqref{conEP} guarantees the boundedness of $\{u_k$\}, which recovers Lemma 8(b) in \cite{2020Convergence}.


\section{Iteration complexity for saddle point problems}\label{se5}
In this section, we first formulate the saddle point problem which we want to solve and then deduce the iteration complexity results.

\subsection{Problem formulation and preliminary results}
Let $f: \RR^m\times \RR^n\rightarrow \RR$ be a given function. We consider finding a saddle point $(\bar{x},\bar{y})$ of the problem
\begin{equation}\label{sadd}
\min_{x\in \RR^m}\max_{y\in \RR^n} f(x,y). \\[4pt]
\end{equation}
In other words, find a pair $(\bar{x},\bar{y})\in \RR^m\times \RR^n$, which will be called saddle point, to satisfy the following relation
$$f(\bar{x}, y)\leq f(\bar{x},\bar{y})\leq f(x,\bar{y})$$
for all $x\in \RR^m, y\in \RR^n$.
Let $z=[x; y]\in \RR^{n+m}$ and define the operator $F: \RR^{n+m}\rightarrow \RR^{n+m}$ as
\begin{equation}\label{oper}
F(z):= [\nabla_xf(x, y); -\nabla_yf(x, y)].
\end{equation}
In order to apply previous theory to this specialized operator $F$,
we restrict our attention to the problem \eqref{sadd} with the following assumption.
\begin{assumption}\label{basic}
The set of all saddle point pairs of the problem \eqref{sadd}, denoted by $\cZ$, is nonempty. Let positive parameters $L_{xx}, L_{xy}, L_{yy}, L_{yx}$ be given and let
    $$L:=2\times \max\{L_{xx}, L_{xy}, L_{yy}, L_{yx}\}.$$  The function $f(x,y)$ in \eqref{sadd} is
\begin{itemize}
  \item continuously differetiable in $x$ and $y$,
  \item convex in $x$ for any fixed $y$ and concave in $y$ for any fixed $x$, and
  \item gradient-Lipschitz-continuous with modulus $L$ in the sense that
   $$\|\nabla_xf(x_1, y)-\nabla_xf(x_2, y)\| \leq L_{xx}\|x_1-x_2\| \quad \textrm{for all}\quad y,$$
   $$\|\nabla_xf(x, y_1)-\nabla_xf(x, y_2)\| \leq L_{xy}\|y_1-y_2\|\quad \textrm{for all}\quad x,$$
      $$\|\nabla_yf(x, y_1)-\nabla_yf(x, y_2)\| \leq L_{yy}\|y_1-y_2\|\quad \textrm{for all}\quad x,$$
         $$\|\nabla_yf(x_1, y)-\nabla_yf(x_2, y)\| \leq L_{yx}\|x_1-x_2\|\quad \textrm{for all}\quad y.$$
\end{itemize}
\end{assumption}
The gradient-Lipschitz-continuity above implies the standard Lipschitzness of $F$, which further implies the relative Lipschitzness of $F$ due to Lemma \ref{resclip}.
Now, by substituting the formula \eqref{oper} for $F$ in the algorithmic frameworks \eqref{alg1} and \eqref{alg2}, the saddle point problems in the form \eqref{sadd} could be solved in some degree. Especially, for the saddle point problems that satisfy Assumption \ref{basic}, we will derive the convergence rate of $\cO(1/k)$ for the BEG and BEP methods. To this end, we let $\{z_k\}$ be the concerned iterate sequence with $z_k:=[x_k;y_k]\in \RR^{n+m}$. Let $\{r_k\}$ be a parameter sequence with $r_k>0$ and let $s_k:=\sum_{i=0}^kr_i$. We denote the averaged iterates of $\{z_k\}$ by $\{r_k\}$ as follows
\begin{equation}\label{aver}
\hat{z}_k:=[\hat{x}_k;\hat{y}_k]:=[\sum_{i=0}^k\frac{r_i}{s_i}x_i; \sum_{i=0}^k\frac{r_i}{s_i}x_i]=\sum_{i=0}^k\frac{r_i}{s_i}z_i.
\end{equation}
In terms of these notations, we have the following preliminary results for the saddle point problem \eqref{sadd}. It should be pointed out that they are well known--see for example \cite{2004Prox}.
\begin{lemma}\label{lempre}
Consider the saddle point problem \eqref{sadd} with Assumption \ref{basic} and recall the definitions in \eqref{oper}-\eqref{aver}. Let $\omega:\RR^d\rightarrow \RR$ be a strongly convex function with modulus $\mu>0$. Then,
\begin{itemize}
\item Assumptions \ref{nonemp}-\ref{mono} and Assumption \ref{Lip} with $\lambda=\frac{L}{\mu}$ hold for the operator $F$, and especially, $\cZ\subset \cY$, i.e., $F(\bar{z})=0$ for any $\bar{z}\in \cZ$;
\item For any $z=[x; y]\in \RR^{n+m}$, we have
\begin{equation}\label{gap}
f(\hat{x}_k,y)-f(x, \hat{y}_k)\leq \frac{1}{s_k}\sum_{i=0}^k r_i\langle F(z_i), z_i-z\rangle.
\end{equation}
\end{itemize}
\end{lemma}
We remark that the gradient Lipschitz continuity in Assumption \ref{basic} is only a sufficient condition guaranteeing the weaker relative Lipschitzness.

\subsection{Iteration complexity results}
The first result is about the iteration complexity of BEG method applied to the saddle point problem with Assumption \ref{basic}. Before presentation, we first clarify some notations.

Let $\{\bar{u}_k\}$ be the iterates generated by the BEG method with initial points $u_0, u^*_0$, $\mu$-strongly convex $\omega$ and $L$-Lipschitz operator $F=[\nabla_xf; -\nabla_yf]$, and the parameters $\{\alpha_k\}$ satisfying $0<L\alpha_k\leq \mu$. To apply the preliminary result in Lemma \ref{lempre}, we let $z_k=\bar{u}_k$ and $r_k=\alpha_k$.
Let $\bar{z}=[\bar{x};\bar{y}]\in \cZ\subset \cY$ be a fixed saddle point. Denote
$$R_1:=\max\{ \max_k\{\|\bar{u}_k\|, \|\bar{z}\|\};$$
then $R_1<+\infty$ due to the boundedness of $\{\bar{u}_k\}$. Let $$\cD_1:=\{z\in \RR^{m+n}: \|z\|\leq R_1\},$$ which is the smallest compact convex set including the iterates $\{\bar{u}_k\}$ (or say $\{z_k\}$) and the saddle point $\bar{z}$. By convexity, it also includes the averaged iterates $\{\hat{z}_k\}$, as defined in \eqref{aver}. In terms of these notations, we have the following iteration complexity result, which recovers the convergence rate of $\cO(1/k)$ by taking the positive parameters $\alpha_k$ to be some constant.

\begin{proposition}\label{resu1}
Let $\omega:\RR^d\rightarrow \RR$ be a strongly convex function with modulus $\mu>0$. The BEG method with $\alpha_k$ satisfying $0<L\alpha_k\leq \mu$ and initial points $u_0, u^*_0$ being given, applied to the saddle point problem \eqref{sadd} with Assumption \ref{basic}, converges sublinearly in the sense that
\begin{equation}\label{bound}
|f(\hat{x}_k, \hat{y}_k)-f(\bar{x},\bar{y})|\leq \max_{z:z\in \cD_1} \frac{1}{\sum_{i=0}^k \alpha_i}D_\omega^{u_0^*}(z,u_0).
\end{equation}
\end{proposition}

The second result is about the iteration complexity of the BEP method applied to the saddle point problem \eqref{sadd}. Since its deduction is similar to that of the first result, we here just briefly point out the difference notations.
Let $\{u_k\}$ be the iterates generated by the BEP method with $F=[\nabla_xf; -\nabla_yf]$, the initial points $u_0=u_{-1}$ being given, and the positive parameters $\alpha_k, \beta_k$ satisfying the condition \eqref{eq0}. Let $z_k=u_{k+1}$ and $r_k=\alpha_k$.
Let $\bar{z}=[\bar{x};\bar{y}]\in \cZ\subset \cY$ be a fixed saddle point. Denote
$$R_2:=\max\{ \max_k\{\|u_k\|, \|\bar{z}\|\};$$
then $R_2<+\infty$ due to the boundedness of $\{u_k\}$. Let $$\cD_2:=\{z\in \RR^{m+n}: \|z\|\leq R_2\},$$ which is the smallest compact convex set including the iterates $\{u_k\}$ (or say $\{z_k\}$) and the saddle point $\bar{z}$. By convexity, it also includes the averaged iterates $\{\hat{z}_k\}$, as defined in \eqref{aver}. In terms of these notations, we present the second iteration complexity result, whose proof follows directly by replacing \eqref{eqc1} with \eqref{eqd3} and repeating the argument for Proposition \ref{resu1}.

\begin{proposition}\label{resu2}
Let $\omega:\RR^d\rightarrow \RR$ be a strongly convex function with modulus $\mu>0$. The BEP method with $\{\alpha_k, \beta_k\}$ satisfying \eqref{eq0} and initial points $u^*_0, u_{-1}=u_0$ being given, applied to the saddle point problem \eqref{sadd} with Assumption \ref{basic}, converges sublinearly in the sense that
\begin{equation}\label{bound2}
|f(\hat{x}_k, \hat{y}_k)-f(\bar{x},\bar{y})|\leq \max_{z:z\in \cD_2} \frac{1}{\sum_{i=0}^k \alpha_i}D_\omega^{u_0^*}(z,u_0).
\end{equation}
\end{proposition}

The iterate complexities above show that the function value of the averaged iterates generated by BEG or BEP converges to the function value at any fixed saddle point of the problem \eqref{sadd}.

\section{Conclusion}\label{se6}
Partially inspired by the elegant results in \cite{2020Convergence}, we introduce two algorithmic frameworks with generalized Bregman distances and study their iteration complexity for solving saddle point problems. With the help of the recent concept of relative Lipschitzness and Bregman distance related tools, our approach is simple and essentially different from the proximal point approach taken in \cite{2020Convergence}. Moreover, our theory is general in the sense that it applies to more general algorithmic frameworks under weaker assumptions. The numerical performance of the sparse and regularized variants definitely deserve further study and we leave it as future work.

\section*{Acknowledgements}
This work is supported by the National Science Foundation of China (No.11971480), the Natural Science Fund of Hunan for Excellent Youth (No.2020JJ3038), and the Fund for NUDT Young Innovator Awards (No. 20190105).

\section*{Appendix: The missing proofs}

\noindent {\bf The proof of Proposition \ref{conEG}}:
From the formula $\bar{u}_k =\nabla \omega^*(u_k^*- \alpha_k F(u_k))$ in \eqref{alg12}, we have
$$u_k^*- \alpha_k F(u_k)\in \partial \omega(\bar{u}_k).$$
 Take $\bar{u}_k^*:=u_k^*- \alpha_k F(u_k)$; then using the fact of $u_k^*\in\partial \omega(u_k)$ and applying  \eqref{Bregdis1} in Lemma \ref{lemBreg}, we derive that
\begin{eqnarray}\label{est1}
\begin{array}{lll}
  \langle \alpha_k F(u_k), \bar{u}_k-u\rangle &=& -\langle u^*_k-\bar{u}_k^*, u-\bar{u}_k\rangle\\
&= &  D_\omega^{u_k^*}(u,u_k) - D_\omega^{\bar{u}_k^*}(u,\bar{u}_k) - D_\omega^{u_k^*}(\bar{u}_k,u_k).
\end{array}
\end{eqnarray}
Similarly, starting with $u_{k+1}^*= u_k^*- \alpha_k F(\bar{u}_k)$ in \eqref{alg12} and applying \eqref{Bregdis1} in Lemma \ref{lemBreg} again, we derive that
\begin{eqnarray}\label{est2}
\begin{array}{lll}
  \langle \alpha_k F(\bar{u}_k), u_{k+1}-u \rangle &=& -\langle u^*_k-u_{k+1}^*, u-u_{k+1}\rangle\\
&= & D_\omega^{u_k^*}(u,u_k) - D_\omega^{u_{k+1}^*}(u,u_{k+1}) - D_\omega^{u_k^*}(u_{k+1},u_k).
\end{array}
\end{eqnarray}
Substituting $u_{k+1}$ for $u$ in \eqref{est1}, we obtain
\begin{equation}\label{est3}
\langle \alpha_k F(u_k), \bar{u}_k-u_{k+1}\rangle=D_\omega^{u_k^*}(u_{k+1},u_k) - D_\omega^{\bar{u}_k^*}(u_{k+1},\bar{u}_k) - D_\omega^{u_k^*}(\bar{u}_k,u_k).
\end{equation}
Combining \eqref{est2} and \eqref{est3}, we have
\begin{eqnarray*}\label{est4}
\begin{array}{lll}
  \langle \alpha_k F(\bar{u}_k), \bar{u}_k-u \rangle &=&  D_\omega^{u_k^*}(u,u_k) - D_\omega^{u_{k+1}^*}(u,u_{k+1}) \\
& &+\alpha_k\langle F(\bar{u}_k)-F(u_k), \bar{u}_k-u_{k+1}\rangle -D_\omega^{u_k^*}(\bar{u}_k,u_k)-D_\omega^{\bar{u}_k^*}(u_{k+1},\bar{u}_k).
\end{array}
\end{eqnarray*}
Invoking the relative Lipschitz assumption and noting $\lambda\alpha_k\leq 1$, we obtain \eqref{recur1}.

Now, we turn to prove the boundedness. Summing up \eqref{recur1} from $k=0$ to $t-1$, we obtain
\begin{equation}\label{eqc1}
\sum_{k=0}^{t-1}\langle \alpha_k F(\bar{u}_k), \bar{u}_k-u \rangle \leq D_\omega^{u_0^*}(u,u_0) - D_\omega^{u_t^*}(u,u_t).
\end{equation}
Using the monotonicity of $F$ in Assumption \ref{mono} and the condition $F(\bar{u})=0$ for any $\bar{u}\in\cY$, we have that
$$\langle  F(\bar{u}_k), \bar{u}_k-\bar{u} \rangle= \langle  F(\bar{u}_k)-F(\bar{u}), \bar{u}_k-\bar{u}\rangle\geq 0.$$
This means that each term in the summand in \eqref{eqc1} with $u=\bar{u}$ is nonnegative and hence implies the following
\begin{equation}\label{eqc3}
D_\omega^{u_t^*}(\bar{u},u_t)\leq D_\omega^{u_0^*}(\bar{u},u_0).
\end{equation}
Note from \eqref{Bregdis2} in Lemma \eqref{lemBreg} that
$$D_\omega^{u_t^*}(\bar{u},u_t)= D_{\omega^*}^{\bar{u}}(u_t^*,\bar{u}^*)=\omega^*(u_t^*)-\omega^*(\bar{u}^*)-\langle \bar{u}, u_t^*-\bar{u}^*\rangle.$$
Thereby, we have
\begin{equation}\label{eqc2}
\omega^*(u_t^*)-\langle \bar{u}, u_t^*\rangle\leq D_\omega^{u_0^*}(\bar{u},u_0)+\omega^*(\bar{u}^*)-\langle \bar{u}, \bar{u}^*\rangle.
\end{equation}
This implies that for the fixed $\bar{u}$, $\{\omega^*(u_t^*)-\langle \bar{u}, u_t^*\rangle\}$ is bounded.
Together with Lemma \ref{coer1} and Assumption \ref{coer}, we conclude that $\{u_t^*\}$ is bounded. Finally, using the strong convexity of $\omega$ and the fact in Lemma \ref{scLip}, we deduce that
$$\|u_t-u_0\| =\|\nabla \omega^*(u_t^*)-\nabla \omega^*(u_t^*)\|\leq \frac{1}{\mu}\|u_t^*-u_0^*\|,$$
from which the boundedness of $\{u_t\}$ immediately follows. This completes the proof.

\bigskip

\noindent {\bf The proof of Proposition \ref{conEP}}:
First, recall from \eqref{alg22} that
$$u_k^*- u_{k+1}^*= \alpha_k F(u_k)+\alpha_k\beta_k(F(u_k)-F(u_{k-1})).$$
It follows that
\begin{equation}\label{eq1}
\langle u_k^*- u_{k+1}^*, u_{k+1}-u\rangle =\langle \alpha_k F(u_k)+\alpha_k\beta_k(F(u_k)-F(u_{k-1})), u_{k+1}-u\rangle.
\end{equation}
On the other hand, applying  \eqref{Bregdis1} in Lemma \ref{lemBreg}, we have
\begin{equation}\label{eq2}
\langle u_k^*- u_{k+1}^*, u_{k+1}-u\rangle =  D_\omega^{u_k^*}(u,u_k) - D_\omega^{u_{k+1}^*}(u,u_{k+1}) - D_\omega^{u_k^*}(u_{k+1},u_k).
\end{equation}
Thus, combining \eqref{eq1} and \eqref{eq2}, we obtain
\begin{equation}\label{eq3}
\langle \alpha_k F(u_k)+\alpha_k\beta_k(F(u_k)-F(u_{k-1})), u_{k+1}-u\rangle =  D_\omega^{u_k^*}(u,u_k) - D_\omega^{u_{k+1}^*}(u,u_{k+1}) - D_\omega^{u_k^*}(u_{k+1},u_k).
\end{equation}
Let $\Delta F_k: =F(u_k)-F(u_{k-1})$; then
\begin{eqnarray}\label{eq4}
\begin{array}{lll}
 &&\langle \alpha_k F(u_k)+ \alpha_k\beta_k(F(u_k)-F(u_{k-1})), u_{k+1}-u\rangle\\
&=&  \alpha_k\beta_k\langle \Delta F_k, u_{k+1}-u\rangle-\alpha_k\langle \Delta F_{k+1}, u_{k+1}-u\rangle+ \alpha_k\langle F(u_{k+1}), u_{k+1}-u\rangle\\
&=& \alpha_k\beta_k\langle \Delta F_k, u_k-u\rangle +\alpha_k\beta_k\langle \Delta F_k, u_{k+1}-u_k\rangle\\
&&-\alpha_k\langle \Delta F_{k+1}, u_{k+1}-u\rangle+ \alpha_k\langle F(u_{k+1}), u_{k+1}-u\rangle.
\end{array}
\end{eqnarray}
Inserting \eqref{eq3} into \eqref{eq4} and rearranging the terms, we have
\begin{eqnarray}\label{eq5}
\begin{array}{lll}
 &&\alpha_k\langle F(u_{k+1}), u_{k+1}-u\rangle \\
 &=&\alpha_k\langle \Delta F_{k+1}, u_{k+1}-u\rangle-\alpha_k\beta_k\langle \Delta F_k, u_k-u\rangle+\alpha_k\beta_k\langle \Delta F_k,u_k- u_{k+1}\rangle\\
 &&+D_\omega^{u_k^*}(u,u_k) - D_\omega^{u_{k+1}^*}(u,u_{k+1}) - D_\omega^{u_k^*}(u_{k+1},u_k).
\end{array}
\end{eqnarray}
Using the relative Lipschitzness assumption, we have that
\begin{eqnarray}\label{eq6}
\begin{array}{lll}
 \langle \Delta F_k,u_k- u_{k+1}\rangle&=&\langle F(u_k)-F(u_{k-1}) ,u_k- u_{k+1}\rangle\\
 &\leq &\lambda D_\omega^{u_k^*}(u_{k+1},u_k)+\lambda D_\omega^{u_{k-1}^*}(u_k,u_{k-1}).
\end{array}
\end{eqnarray}
Finally, combining \eqref{eq5} and \eqref{eq6} and noting \eqref{eq0}, we obtain \eqref{eq01}.

Now, we turn to prove the boundedness. Summing up \eqref{eq01} from $k=0$ to $t-1$, we obtain
\begin{eqnarray}\label{eqd1}
\begin{array}{lll}
 \sum_{k=0}^{t-1}\alpha_k\langle F(u_{k+1}), u_{k+1}-u\rangle&\leq & \alpha_{t-1}\langle \Delta F_t, u_t-u\rangle-\alpha_{-1}\langle \Delta F_0, u_0-u\rangle\\
 &&+D_\omega^{u_0^*}(u,u_0) - D_\omega^{u_t^*}(u,u_t)\\
 &&+\lambda\alpha_{-1}D_\omega^{u_{-1}^*}(u_0,u_{-1}) - \lambda\alpha_{t-1}D_\omega^{u_{t-1}^*}(u_t,u_{t-1}).
\end{array}
\end{eqnarray}
Using the relative Lipschitzness assumption, we have that
\begin{eqnarray}\label{eqd2}
\begin{array}{lll}
 \langle \Delta F_t,u_t- u\rangle&=&\langle F(u_t)-F(u_{t-1}) ,u_t- u\rangle\\
 &\leq &\lambda D_\omega^{u_t^*}(u,u_t)+\lambda D_\omega^{u_{t-1}^*}(u_t,u_{t-1}).
\end{array}
\end{eqnarray}
Combining \eqref{eqd1} and \eqref{eqd2} and noting $\|\Delta F_0\|=D_\omega^{u_0^*}(u,u_0)=0$ due to the fact $u_0=u_{-1}$, we obtain
\begin{equation}\label{eqd3}
\sum_{k=0}^{t-1}\alpha_k\langle F(u_{k+1}), u_{k+1}-u\rangle\leq D_\omega^{u_0^*}(u,u_0)-(1-\lambda \alpha_{t-1})D_\omega^{u_t^*}(u,u_t).
\end{equation}
Using \eqref{eqd3} with $u=\bar{u}\in \cY$ and repeating the argument below \eqref{eqc1}, we deduce that
\begin{equation}\label{eqd4}
(1-\lambda \alpha_{t-1})D_\omega^{u_t^*}(\bar{u},u_t)\leq D_\omega^{u_0^*}(\bar{u},u_0).
\end{equation}
Repeating the argument below \eqref{eqc3} and noting that $1-\lambda \alpha_{t-1}>\rho$, we finally conclude that the sequence $\{u_k\}$ is bounded.
This completes the proof.

\bigskip

\noindent {\bf The proof of Proposition \ref{resu1}}:
Combining \eqref{gap} in Lemma and \eqref{eqc1}, we obtain
\begin{equation}\label{eqe1}
 f(\hat{x}_k,y)-f(x, \hat{y}_k)\leq \frac{1}{s_k}\sum_{i=0}^k r_i\langle F(z_i), z_i-z\rangle\leq \frac{1}{s_k}D_\omega^{u_0^*}(z,u_0).
\end{equation}
In view of the definition of $\cD_1$, we derive that
\begin{eqnarray}\label{eqe2}
\begin{array}{lll}
& &\max_{y:[\hat{x}_k; y]\in \cD_1} f(\hat{x}_k, y)- \min_{x:[x;\hat{y}_k]\in \cD_1} f(x,\hat{y}_k)\\
 &= &\max_{z:z\in \cD_1} \left[f(\hat{x}_k, y)- f(x,\hat{y}_k)\right]\\
 &\leq &\max_{z:z\in \cD_1} \frac{1}{s_k}D_\omega^{u_0^*}(z,u_0).
\end{array}
\end{eqnarray}
Note that
$$\max_{y:[\hat{x}_k; y]\in \cD_1} f(\hat{x}_k, y)\geq f(\hat{x}_k, \bar{y})\geq f(\bar{x},\bar{y})$$
and
$$\min_{x:[x;\hat{y}_k]\in \cD_1} f(x,\hat{y}_k)\leq f(\bar{x},\hat{y}_k)\leq f(\bar{x},\bar{y})$$
due to the fact of $\bar{z}=[\bar{x};\bar{y}]\in\cD_1$ and the definition of saddle points. Thus, together with the fact of $[\hat{x}_k;\hat{y}_k]\in\cD_1$ and using \eqref{eqe2}, we derive that
$$f(\hat{x}_k, \hat{y}_k)-f(\bar{x},\bar{y})\leq \max_{y:[\hat{x}_k; y]\in \cD_1} f(\hat{x}_k, y) -\min_{x:[x;\hat{y}_k]\in \cD_1} f(x,\hat{y}_k)\leq \max_{z:z\in \cD_1} \frac{1}{s_k}D_\omega^{u_0^*}(z,u_0)$$
and
$$f(\bar{x},\bar{y})-f(\hat{x}_k, \hat{y}_k)\leq \max_{y:[\hat{x}_k; y]\in \cD_1} f(\hat{x}_k, y) -\min_{x:[x;\hat{y}_k]\in \cD_1} f(x,\hat{y}_k) \leq \max_{z:z\in \cD_1} \frac{1}{s_k}D_\omega^{u_0^*}(z,u_0).$$
Therefore, the sublinear convergence \eqref{bound} follows.

\end{document}